\documentclass[12pt]{amsart}

\usepackage{amssymb,latexsym,amsmath,extarrows, mathrsfs, amsthm,color,amsrefs}
\usepackage{graphicx}

\usepackage[margin=1in, centering]{geometry}
\usepackage[bookmarksnumbered, colorlinks, plainpages]{hyperref}
\usepackage{hyperref} 
\hypersetup{
    colorlinks=true,       
    linkcolor=blue,          
    citecolor=magenta,        
    filecolor=magenta,      
    urlcolor=cyan           
}

\usepackage{mathtools}
\mathtoolsset{showonlyrefs,showmanualtags}

\newtheorem{theorem}{Theorem}[section]

\newtheorem{lemma}[theorem]{Lemma}

\newtheorem{remark}[theorem]{Remark}

\newtheorem*{definition}{Definition}
\newtheorem{corollary}[theorem]{Corollary}

\newcommand{\C}{\mathbb C}
\newcommand{\Z}{\mathbb Z}
\newcommand{\T}{\mathbb T}
\newcommand{\N}{\mathbb N}
\newcommand{\R}{\mathbb R}

\newcommand{\ol}{\overline}
\newcommand{\D}{\mathbb D}
\newcommand{\Q}{\mathbb Q}

\title[Localization]{Localization for magnetic quantum walks}
\author{Fan Yang}
\address{Department of Mathematics, Louisianan State University, Baton Rouge, LA 70803}
\begin{document}

\begin{abstract} 
We prove Anderson localization for all Diophantine frequencies and all non-resonant phases for a model that arises from 2D quantum walk model subject to an external magnetic field. 
This is the first localization result for all Diophantine frequencies in the magnetic quantum walk and the quasi-periodic CMV setting. 
We also obtain sharp asymptotics of the localized eigenfunctions.
\end{abstract}

\maketitle

\section{Introduction}

Quantum walk was first proposed by Aharonov, Davidovich and Zagury \cite{ADZ}. It can be viewed as a quantum mechanical analogue of the classical random walk.
Compared to the diffusive transport in the classical random walk, quantum walk leads to a ballistic spreading of the particle's wave function. 
This fast spreading property has played a pivotal role in the development of quantum algorithms \cite{Amb07, Por2013}, including search algorithms, element distinctness and matrix product verification.
Besides the applications in quantum information science, quantum walks are also very accessible and interesting to both experimental and theoretical studies for many complex quantum phenomena in physics. We refer interested readers to \cite{VA12,Kon08} for a comprehensive review on quantum walks.

In recent years, there have been an growing interest on quantum walk in mathematical community, see e.g. \cite{BGVW, CGMV, CWloc,DFO,JM,HJ,J}.
In particular, in \cite{CGMV}, the authors discovered a beautiful connection between quantum walks and the CMV matrices,
which is a class of unitary operators which arise in the theory of orthogonal polynomials on the unit circle (OPUC) \cite{SOPUC,5years}.
Recently, quantum walk model in electric fields also attracts a lot of attention in physics, see e.g. \cite{GASWWMA,CRWAGW,WLKGB}, and Anderson localization for that model was proved in \cite{CWloc} for a.e. electric field. In an upcoming work \cite{ew}, we prove localization for electric quantum walks for all Diophantine fields.

Now I will introduce the model that we study, which was given recently in \cite{CFO} as a generalization of the model studied in \cite{FOZ}. 
This model arises from the two dimensional quantum walk on $\Z^2$ subject to a homogeneous magnetic field, see Section 3 of \cite{CFO}.

Let $W: \ell^2(\Z)\otimes \C^2\to \ell^2(\Z)\otimes \C^2=:\mathcal{H}$ be a quantum walk defined by a quasi-periodic sequence of coins. 
We denote the standard basis of $\mathcal{H}$
\begin{align*}
\delta_n^s=\delta_n\otimes e_s,\ \  n\in \Z,\  s\in \{+,-\},
\end{align*} 
where $\{\delta_n\}$ is the standard basis of $\ell^2(\Z)$ and $\{e_+=(1,0)^T,\ \ e_-=(0,1)^T\}$ is the standard basis of $\C^2$.
Given coupling constants $\lambda_1,\lambda_2\in [0,1]$, let $\lambda_j':=\sqrt{1-\lambda_j^2}$, $j=1,2$.
Let frequency $\omega\in \T$ and phase $\theta\in \T$, we consider the following operator acting on $\mathcal{H}_1$,
\begin{align}\label{def:W}
W_{\lambda_1,\lambda_2,\omega,\theta}:=S_{\lambda_1} Q_{\lambda_2,\omega,\theta},
\end{align}
where $Q_{\lambda_2,\omega,\theta}$ acts coordinate-wise via $Q_{\lambda_2,\omega,\theta,n}$ defined by 
\begin{align*}
Q_{\lambda_2,\omega,\theta,n}=
\left(\begin{matrix}
\lambda_2\cos(2\pi(\theta+n\omega))+i\lambda_2' & -\lambda_2\sin(2\pi(\theta+n\omega))\\
\lambda_2\sin(2\pi(\theta+n\omega)) &\lambda_2\cos(2\pi(\theta+n\omega))-i\lambda_2'
\end{matrix}\right)
\end{align*}
and 
\begin{align*}
S_{\lambda_1}\delta_n^{\pm}=\lambda_1\delta_{n\pm 1}^{\pm}\pm \lambda_1'\delta_n^{\mp}.
\end{align*}

This model is called the unitary almost Mathieu operator (UAMO) due to its close connection to the celebrated almost Mathieu operator. 
In a recent work \cite{CFO}, the authors generalized the shift operator from $S_1$ in \cite{FOZ} to $S_{\lambda_1}$, which leads to a spectral transition phenomena in the unitary setting. 
In particular, \cite{CFO} proves Anderson localization for almost every frequency $\omega$ and phase $\theta$ in the positive Lyapunov exponent regime. 
Their result is in the measure theoretical setting, hence is inconclusive for {\it any} given Diophantine frequency and {\it any} given phase.

In this paper, we improve on this result by proving localization for all Diophantine frequencies \eqref{def:Dio}, and a.e. $\theta$ which is arithmetically characterized by \eqref{eq:theta_set}.
This proof is highly inspired by Jitomirskaya's work \cite{jit99} on the almost Mathieu operator. 
The approach we present here makes the study of many central topics of the unitary almost Mathieu operator possible, including the (dry) Ten Martini problem \cite{Puig,AJ1,AJ2}, sharp arithmetic spectral transitions \cite{jit99,AYZ,JL}, eigenfunction asymptotics and hierarchy \cite{JL,JL2}, and sharp H\"older continuity of the Lyapunov exponent, density of states and spectral measures \cite{AJ2,AJ3}.

This model fits into the framework of generalized extended CMV matrix \cite{CGMV}, and it is more convenient for us to work within that setting. Let us review the connection below.

For notational convenience in the rest of this paper we identify $\ell^2(\Z)\otimes \C^2 \to \ell^2(\Z)$ via
$$\delta_n^+ \mapsto \delta_{2n},\ \ \delta_n^-\mapsto \delta_{2n+1}.$$

For $n\in \Z$, let 
\begin{align}\label{def:alpha_n}
\begin{cases}
\alpha_{2n}:=\lambda_1' \text{ and } \rho_{2n}:=\lambda_1\\
\alpha_{2n+1}:=\lambda_2\sin(2\pi(\theta+n\omega)) \text{ and } \rho_{2n+1}:=\lambda_2\cos(2\pi(\theta+n\omega))+i\lambda_2'
\end{cases}
\end{align}
It is clear that $|\alpha_n|^2+|\rho_n|^2=1$.
Let 
\begin{align}\label{def:Theta}
\Theta_n:=\left(\begin{matrix}
\ol{\alpha}_n &\rho_n\\
\ol{\rho}_n &-\alpha_n
\end{matrix}\right).
\end{align}
Note that in the standard CMV setting, the lower-left entry of $\Theta_n$ is a real-valued $\rho_n$ instead of our complex-valued $\ol{\rho}_n$.
Hence this model is a generalization of the the CMV matrix.

\begin{align}\label{def:LM}
\mathcal{L}_{\lambda_1}=\bigoplus \Theta_{2n}, \text{ and } \mathcal{M}_{\lambda_2,\omega,\theta}:=\bigoplus \Theta_{2n+1},
\end{align}
where $\Theta_{n}$ acts on $\ell^2(\{n, n+1\})$.
Then the operator $W$ can be written as follows:
\begin{align}\label{eq:W=LM}
W_{\lambda_1,\lambda_2,\omega,\theta}=\mathcal{L}_{\lambda_1}\mathcal{M}_{\lambda_2,\omega,\theta}.
\end{align}
This is an extended CMV matrix, since it is a whole-line operator. It is clear that the spectrum $\sigma(W)$ of $W$ is contained in $\partial \D$.

Now we state the main result of this paper. Let $\|x\|_{\T}=\mathrm{dist}(x,\Z)$ be the torus norm.
We say $\omega$ is Diophantine, denoted by $\omega\in \mathrm{DC}$, if 
\begin{align}\label{def:Dio}
\omega\in \cup_{\tau>1}\cup_{\nu>0} \mathrm{DC}(\tau,\nu),
\end{align}
where
\begin{align*}
\mathrm{DC}(\tau,\nu):=\{\omega\in \T:\ \|n\omega\|_{\T}\geq \frac{\nu}{|n|^{\tau}}, \ \text{ for any } 0\neq n\in \Z\}.
\end{align*}
\begin{theorem}\label{thm:main}
For $0\leq \lambda_1<\lambda_2\leq 1$, for all Diophantine frequencies $\omega$ and $\theta$ such that
\begin{align}\label{eq:theta_set}
\tilde{\gamma}(\omega,\theta):=\limsup_{n\to \infty} -\frac{\ln\|2\theta-\frac{1}{2}+n\omega\|_{\T}}{|n|}=0
\end{align} 
$W_{\lambda_1,\lambda_2,\omega,\theta}$ has Anderson localization, namely pure point spectrum with exponentially decaying eigenfunctions. Furthermore, let $\Psi$ be such an exponentially decaying eigenfunction, we have its sharp asymptotic as follows
\begin{align}\label{eq:decay_rate}
\lim_{|y|\to \infty} \frac{\ln (|\Psi_y|^2+|\Psi_{y+1}|^2)}{2|y|}=-\frac{1}{2}\ln\frac{\lambda_2(1+\lambda_1')}{\lambda_1(1+\lambda_2')},
\end{align}
where the decay rate above is the Lyapunov exponent of the associated Szeg\"o cocycle.
\end{theorem}
It is well known that $a.e.\, \omega$ is Diophantine, and $\theta$ satisfying \eqref{eq:theta_set} (the non-resonant phases) is a full measure set.
One should compare our result to Theorem 2.2 (b) of \cite{CFO}, where localization was proved for a.e. $\omega,\theta$ only in the measure-theoretic setting. 
To the best of knowledge, our result is the first localization result for all Diophantine frequencies in the CMV setting and the quantum walk setting, see \cite{WD,CWloc,CFO}.
The sharp asymptotic \eqref{eq:decay_rate} also adds to the recent growing collection of exact characterizations of quasi-periodic eigenfunctions \cite{JL,JL2,ehmtransition,HJYMaryland,ew}.

The key to our proof is Lemma \ref{lem:main}, where we show $P_{[1,2n],z}=\det(z-\mathcal{L}\mathcal{M})|_{[1,2n]}$ is a polynomial of $\sin(2\pi(\theta+\frac{n-1}{2}\omega))$ of degree at most $n$.
A result of this kind was the key to the study of the almost Mathieu operator in \cite{jit99}, and it laid the foundation for most of the recent breakthroughs \cite{jit99,Puig,AJ1,AJ2,JL} on the almost Mathieu operator.
Our paper makes the study of all these problems for the unitary almost Mathieu operator possible.

Next let us comment on the difficulties of the proof of Lemma \ref{lem:main}. 
To show $P_{[1,2n],z}$ is a polynomial of $\sin$, we exploit the symmetries of the matrix $(z-\mathcal{L}\mathcal{M})|_{[1,2n]}$.
Here the symmetries are more complicated than the almost Mathieu operator (or the extended Harper setting \cite{JKS,Handry}), since our matrix is a block Jacobi matrix rather than a scalar Schr\"odinger (Jacobi) operator. 
Furthermore, at a first glance, one might expect that $P_{[1,2n],z}$ is a polynomial of degree $2n$, since it is the determinant of a $2n\times 2n$ matrix whose non-zero entries are all polynomials of degree $1$.
Here we prove that the degree is only half of $2n$, which plays an important role in our proof of Theorem \ref{thm:main}.

Our proof utilizes the equivalence between the Szeg\"o cocycle, the Gesztesy--Zinchenko cocycle and the standard quasi-periodic cocycle in Sections \ref{sec:Sze}, \ref{sec:GZ} and \ref{sec:qs_cocycle}. 
The equivalence between the Gesztesy--Zinchenko cocycle and the Szeg\"o cocycle can be found in \cite{DFLY}. 
To the best of our knowledge, the equivalence between the quasi-periodic cocycle and the Szeg\"o or Gesztesy--Zinchenko cocycle is missing in the literature. We prove it in Section \ref{sec:equiv}. 
In our proof of Theorem \ref{thm:main} we use this equivalence to link the Lyapunov exponent of the standard quasi-periodic cocycle, which was computed recently in \cite{CFO}, to the Lyapunov exponent of the Szeg\"o cocycle, which is closely related to the Green's function and Anderson localization.
We point out that it is possible to directly compute the Lyapunov exponent of the Szeg\"o cocycle directly using Avila's global theory \cite{global}.

The rest of the paper is organized as follows: Section 2 serves as preliminary, Section 3 contains the key Lemma \ref{lem:main}, and Section 4 is devoted to the proof of Theorem \ref{thm:main}.

\section*{Acknowledgement}
This research is partially supported by the AMS Simons Travel grants. I would like to thank Christopher Cedzich and Jake Fillman for useful comments on an earlier version of this paper.

\section{Preliminary}
Let $\mathbb D$ be the unit disk in $\C$ and $\partial \mathbb D$ be the unit circle.
For an matrix $A$, let $A|_{[a,b]}:=\chi_{[a,b]}\, A\, \chi_{[a,b]}$ be the restriction of the matrix to the box $[a,b]$.
Let $\mathcal{L}$ and $\mathcal{M}$ be defined as in \eqref{def:LM} with generally defined $\alpha_n$, and $W=\mathcal{L}\mathcal{M}$. 
Then there holds that $W|_{[a,b]}=\mathcal{L}|_{[a,b]} \mathcal{M}|_{[a,b]}$, see \cite{Kruger}.

\subsection{Orthogonal polynomials and Szeg\"o cocycle}\label{sec:Sze}
Let $\mu$ be a probability measure on $\partial\mathbb D$, supported on an infinite subset of $\partial \D$. 
Let $\Phi_n$ be the monic polynomial of degree $n$ \linebreak (in $z$) such that
\begin{align*}
\int_{\partial \D} \Phi_n(z)\Phi_m(z)\, dz=\delta_{nm}, \text{ for any } m,n\in \N.
\end{align*}
The polynomials $\{\Phi_n\}_{n\in \N}$ are called the orthogonal polynomials on the unit circle with respect to $\mu$.
These polynomials satisfy the following relation, called Szeg\"o recurrence, see \cite{SOPUC},
\begin{align*}
\Phi_{n+1}(z)=&z\Phi_n(z)-\ol{\alpha}_n \Phi_n^*(z)\\
\Phi_{n+1}^*(z)=&\Phi_n^*(z)-\alpha_n z\Phi_n(z),
\end{align*}
where $\Phi_k^*(z)=z^k \ol{\Phi_k(1/\ol{z})}$. The $\alpha_n$'s above are called Verbluncky coefficients.
The following theorem is well-known.
\begin{theorem}[Verblunsky's theorem, see \cite{SOPUC}]
There is a bijection between probability measure supported on an infinite subset of $\partial \D$ and $\{\alpha_n\}_{n\in \N}\subset \D$.
\end{theorem}
Let $\phi_n(z):=\Phi_n(z)/\|\Phi_n(z)\|$ and $\phi_n^*(z):=\Phi_n^*(z)/\|\Phi_n^*(z)\|$.
\begin{align*}
\|\Phi_n(z)\|=|\rho_{n-1}\cdots \rho_0|.
\end{align*}
Hence
\begin{align*}
|\rho_n|\phi_{n+1}(z)=&z\phi_n(z)-\ol{\alpha}_n \phi_n^*(z)\\
|\rho_n|\phi_{n+1}^*(z)=&\phi_n^*(z)-\alpha_n z\phi_n(z).
\end{align*}
Rewriting it in the following way
\begin{align*}
\left(\begin{matrix}
\phi_{n+1}(z)\\
\phi_{n+1}^*(z)
\end{matrix}\right)
=\frac{1}{|\rho_n|}
\left(\begin{matrix}
z\ \ &-\ol{\alpha}_n\\
-\alpha_n z &1
\end{matrix}\right)
\left(\begin{matrix}
\phi_n(z)\\
\phi_n^*(z)
\end{matrix}\right)
=:S_{n,z}
\left(\begin{matrix}
\phi_n(z)\\
\phi_n^*(z)
\end{matrix}\right).
\end{align*}
It is more convenient to consider the normalized two step Szeg\"o matrix:
\begin{align}\label{eq:SS}
\frac{1}{z}S_{n+1,z}S_{n,z}=\frac{1}{|\rho_{n+1}\rho_n|}
\left(\begin{matrix}
\alpha_n \ol{\alpha}_{n+1}+z &-\ol{\alpha}_n-\ol{\alpha}_{n+1}z^{-1}\\
-\alpha_n -\alpha_{n+1}z &\ol{\alpha}_n\alpha_{n+1}+z^{-1}
\end{matrix}\right).
\end{align}
In our setting, let
\begin{align}\label{def:tS}
\tilde{S}(\theta,z):=\frac{1}{z}S_{1,z}S_{0,z}.
\end{align}
We call the cocycle $(\omega, \tilde{S}(\cdot,z))$ the two step Szeg\"o cocycle, it acts on $\T\times \C^2$ in the following way:
\begin{align*}
(\omega, \tilde{S}(\cdot,z)):\ (\theta, v)\mapsto (\theta+\omega, \tilde{S}(\theta,z) v).
\end{align*}
Let $\tilde{S}_{n}(\theta,z):=\tilde{S}(\theta+(n-1)\omega,z)\cdots \tilde{S}(\theta,z)$.
We define the Lyapunov exponent of the two step Szeg\"o cocycle,
\begin{align}\label{def:LS}
L(\omega, \tilde{S}(\cdot,z)):=\lim_{n\to\infty} \frac{1}{n} \int \ln \|\tilde{S}_n(\theta,z)\|\, d\theta.
\end{align}
This limit exist due to Kingman's subadditive ergodic theorem.

\subsection{Gesztesy--Zinchenko cocycle}\label{sec:GZ}
Let $W=\mathcal{L}\mathcal{M}$ be a generalized extended CMV matrix, defined by the Verblunsky coefficients $\alpha_n$.
Suppose $u,v$ are defined as 
\begin{align}\label{def:uv}
Wu=zu, \text{ and } v=\mathcal{L}^{-1}u.
\end{align}
The $(u_n, v_n)^T$ obeys the Gesztesy-Zinchenko iterations \cite{GZ}:
\begin{align}\label{eq:GZ_transfer}
\left(\begin{matrix}
u_{n+1}\\
v_{n+1}
\end{matrix}\right)=
M_{n,z}
\left(\begin{matrix}
u_{n}\\
v_{n}
\end{matrix}\right),
\end{align}
where 
\begin{align*}
M_{n,z}=\frac{1}{\rho_n}
\begin{cases}
\left(\begin{matrix}
-\ol{\alpha}_n & z\\
z^{-1} &-\alpha_n
\end{matrix}\right) \text{ if } n \text{ is odd}\\
\\
\left(\begin{matrix}
-\alpha_n & 1\\
1 &-\ol{\alpha}_n
\end{matrix}\right) \text{ if } n \text{ is even}
\end{cases}
\end{align*}

The following equivalence can be found in \cite{DFLY}: for odd $n$, 
\begin{align}\label{eq:MM=SS}
M_{n+1,z}M_{n,z}=&\frac{1}{\rho_{n+1}\rho_n}\left(\begin{matrix}
\ol{\alpha}_n \alpha_{n+1}+z^{-1} &-\alpha_n-\alpha_{n+1} z\\
-\ol{\alpha}_n-\ol{\alpha}_{n+1} z^{-1} &\alpha_n\ol{\alpha}_{n+1}+z
\end{matrix}\right)\\
=&\frac{|\rho_n \rho_{n+1}|}{z \rho_n \rho_{n+1}} R^{-1} S_{n+1,z} S_{n,z} R,
\end{align}
where 
\begin{align*}
R=\left(\begin{matrix}
0 &1\\
1 &0
\end{matrix}\right).
\end{align*}
In our setting, we denote
\begin{align*}
\tilde{M}(\theta,z):=M_{2,z}M_{1,z}.
\end{align*}

\subsection{Standard quasi-periodic cocycle}\label{sec:qs_cocycle}
Let $u$ be a solution to CMV eigenvalue equation $Wu=zu$. 
We can rewrite this equation in the following way:
\begin{align}\label{def:A}
\left(\begin{matrix}
u_{2n+1}\\
u_{2n}
\end{matrix}\right)
=A_{n,z}
\left(\begin{matrix}
u_{2n-1}\\
u_{2n-2}
\end{matrix}\right),
\end{align}
where
\begin{align*}
A_{n,z}=\frac{1}{\rho_{2n}\rho_{2n-1}}\left(\begin{matrix}
z^{-1}+\alpha_{2n}\ol{\alpha}_{2n-1}+\alpha_{2n-1}\ol{\alpha}_{2n-2}+\alpha_{2n}\ol{\alpha}_{2n-2}z\ \ \ & -\ol{\rho}_{2n-2}\alpha_{2n-1}-\ol{\rho}_{2n-2}\alpha_{2n}z\\
-\rho_{2n}\ol{\alpha}_{2n-1}-\rho_{2n}\ol{\alpha}_{2n-2}z &\rho_{2n}\ol{\rho}_{2n-2}z
\end{matrix}\right)
\end{align*}
For our $\alpha_n$ in \eqref{def:alpha_n}, let 
\begin{align*}
A(\theta,z):=A_{0,z}.
\end{align*}

\subsection{Lyapunov exponent of UAMO}
\begin{theorem}[Theorem 2.7 of \cite{CFO}]\label{thm:LE}
For $z\in \sigma(W_{\lambda_1,\lambda_2,\omega,\theta})$, we have
\begin{align*}
L(\omega, A(\cdot, z))=\max\left(0, \ln \frac{\lambda_2(1+\lambda_1')}{\lambda_1(1+\lambda_2')}\right).
\end{align*}
\end{theorem}
One should note that the Lyapunov exponent is independent of $z$, and is positive iff $0\leq \lambda_1<\lambda_2\leq 1$.

The following integral was computed in \cite{CFO} as well.
\begin{align}\label{eq:int_rho}
\int_{\T}\ln |\lambda_1(\lambda_2\cos(2\pi\theta)+i\lambda_2')|\, d\theta=\ln\frac{\lambda_1(1+\lambda_2')}{2}.
\end{align}

The following control of the norm of the transfer matrix of a continuous cocycle by the Lyapunov exponent is well known.
\begin{lemma}\label{lem:upperbounds}$\mathrm{(}$e.g. \cite{Furman, JMavi}$\mathrm{)}$
Let $(\omega, D)$ be a continuous cocycle, then for any $\varepsilon>0$, for $|k|$ large enough,
\begin{align*}
\|D_k(\theta)\|\leq e^{|k|(L(\omega, D)+\varepsilon)}\ \mathrm{for}\ \mathrm{any}\ \theta\in\T,
\end{align*}
where $D_k(\theta):=\prod_{j=k-1}^0 M(\theta+j\omega)$.
\end{lemma}

\begin{remark}
Considering 1-dimensional continuous cocycles, a corollary of Lemma \ref{lem:upperbounds} is that if $g$
is a continuous function such that $\ln |g| \in L^1(\T)$, 
then for any $\varepsilon > 0$, and $b-a$ sufficiently large, 
\begin{align}\label{eq:scalar_upper}
|\prod_{j=a}^b g (\theta+j\omega) | \leq e^{(b-a+1)(\int_{\T} \ln|g|\, d\theta+\varepsilon)}.
\end{align}
\end{remark}

\subsection{Continued fraction}\label{seccontinued}
Let $\omega\in \T\setminus \mathbb Q$, $\omega$ has the following unique expression with $a_n\in \N$: 
\begin{align*}
\omega=\frac{1}{a_1+\frac{1}{a_2+\frac{1}{a_3+\cdots}}}.
\end{align*}
Let 
\begin{align}\label{defpnqn}
\frac{p_n}{q_n}=\frac{1}{a_1+\frac{1}{a_2+\frac{1}{\cdots+\frac{1}{a_n}}}}
\end{align}
be the continued fraction approximants of $\omega$. 


\subsection{Trigonometric product}

The following lemma from \cite{AJ1} gives a useful estimate of products appearing in our analysis.
\begin{lemma} \label{lana}
Let $\omega\in \R\setminus \Q $, $\theta\in\R$ and $0\leq j_0 \leq q_{n}-1$ be such that 
$$|\cos \pi(\theta+j_{0}\omega)|=\inf_{0\leq j \leq q_{n}-1} |\cos \pi(\theta+j\omega)|,$$
then for some absolute constant $C$,
$$-C\ln q_{n} \leq \sum_{j=0,\, j\neq j_0}^{q_{n}-1} \ln |\cos \pi (\theta+j\omega)|+(q_{n}-1)\ln2 \leq C\ln q_n$$
\end{lemma}

\subsection{Lagrange interpolation theorem}
Let $P_{2n}(\theta)$ be a polynomial in $\sin(2\pi(\theta+\frac{n-1}{2}\omega))$ of degree at most $n$.
For any $\theta_1,\theta_2,...,\theta_{n+1}\in \T$, we have
\begin{align}\label{eq:Lagrange}
|P_{2n}(\theta)| = \left|\sum_{j=1}^{n+1} P_{2n}(\theta_j) \prod_{\substack{\ell=1\\ \ell \neq j}}^{n+1} \frac{z-\sin 2\pi (\theta_{\ell}+\frac{n-1}{2}\omega)}{(\sin 2\pi (\theta_j+\frac{n-1}{2}\omega)-\sin 2\pi (\theta_{\ell}+\frac{n-1}{2}\omega))}\right|.
\end{align}

\subsection{Uniformity}
\begin{definition}
We say that a set $\{\theta_1,\theta_2,...,\theta_{n+1}\}$ is $\kappa$-uniform if 
\begin{align*}
\max_{z\in [0,1]}\, \, \max_{j=1,\cdots,n+1} \left|\prod_{\substack{\ell=1\\ \ell \neq j}}^{n+1} \frac{z-\cos 2\pi (\theta_{\ell}+\frac{n-1}{2}\omega)}{(\cos 2\pi (\theta_j+\frac{n-1}{2}\omega)-\cos 2\pi (\theta_{\ell}+\frac{n-1}{2}\omega))}\right|<e^{\kappa n}.
\end{align*}
\end{definition}

\subsection{Green's function expansion}
Our definition of Green's function is slightly different from those in the literature \cite{Kruger}. In this section, we develop the estimates for the Green's function in our setting.

Let $P_{[a,b],z}:=\det(z-W|_{[a,b]})$.
Let 
\begin{align*}
\mathcal{A}_{[a,b],z}:=
z(\mathcal{L}|_{[a,b]})^{-1}-\mathcal{M}|_{[a,b]}, \text{ and } G_{[a,b],z}=(\mathcal{A}_{[a,b],z})^{-1},
\end{align*}
and
\begin{align*}
\tilde{\mathcal{A}}_{[a,b],z}:=z(\mathcal{M}|_{[a,b]})^{-1}-\mathcal{L}|_{[a,b]}, \text{ and } \tilde{G}_{[a,b],z}=(\tilde{\mathcal{A}}_{[a,b],z})^{-1}.
\end{align*}
Then it is clear that
\begin{align}\label{eq:P=detAdetL}
P_{[a,b],z}&=\det(\mathcal{L}|_{[a,b]}) \det(\mathcal{A}_{[a,b],z})\\
P_{[a,b],z}&=\det(\mathcal{M}|_{[a,b]}) \det(\tilde{\mathcal{A}}_{[a,b],z})
\end{align}
For odd $y\in [a,b]$, we have by Cramer's rule that
\begin{align}\label{eq:Cramer1}
\ \ |G_{[a,b],z}(y,a)|=&(\prod_{j=a}^{y-1}|\rho_j|) \frac{|\det(\mathcal{A}_{[y+1,b],z})|}{|\det(\mathcal{A}_{[a,b],z})|}=
(\prod_{j=a}^{y-1}|\rho_j|)\frac{|P_{[y+1,b],z}|}{|P_{[a,b],z}|} \frac{|\det(\mathcal{L}|_{[a,b]})|}{|\det(\mathcal{L}|_{[y+1,b]})|}, \text{ and }\\
|G_{[a,b],z}(y,b)|=&(\prod_{j=y}^{b-1}|\rho_j|)\frac{|\det(\mathcal{A}_{[a,y-1],z})|}{|\det(\mathcal{A}_{[a,b],z})|}=
(\prod_{j=y}^{b-1}|\rho_j|)\frac{|P_{[a,y-1],z}|}{|P_{[a,b],z}|} \frac{|\det(\mathcal{L}|_{[a,b]})|}{|\det(\mathcal{L}|_{[a,y-1]})|}. \notag
\end{align}
For even $y\in [a,b]$, similarly, we have
\begin{align}\label{eq:Cramer2}
|\tilde{G}_{[a,b],z}(y,a)|=&(\prod_{j=a}^{y-1}|\rho_j|)\frac{|\det(\tilde{\mathcal{A}}_{[y+1,b],z})|}{|\det(\tilde{\mathcal{A}}_{[a,b],z})|}
=(\prod_{j=a}^{y-1}|\rho_j|)\frac{|P_{[y+1,b],z}|}{|P_{[a,b],z}|}\frac{|\det(\mathcal{M}|_{[a,b]})|}{|\det(\mathcal{M}|_{[y+1,b]})|}, \text{ and }\\
|\tilde{G}_{[a,b],z}(y,b)|=&(\prod_{j=y}^{b-1}|)\rho_j|\frac{|\det(\tilde{\mathcal{A}}_{[a,y-1],z})|}{|\det(\tilde{\mathcal{A}}_{[a,b],z})|}
=(\prod_{j=y}^{b-1}|\rho_j|)\frac{|P_{[a,y-1],z}|}{|P_{[a,b],z}|} \frac{|\det(\mathcal{M}|_{[a,b]})|}{|\det(\mathcal{M}|_{[a,y-1]})|}. \notag
\end{align}

Let $\Psi$ be a solution to $W\Psi=z\Psi$, we have the following Poisson formula, for $y\in [a,b]$:
\begin{align}\label{eq:Poisson}
\Psi_y
=&-G_{[a,b],z}(y,a)\cdot \begin{cases}
z(\frac{1}{\alpha_{a-1}}-\ol{\alpha}_{a-1})\Psi_a+z\rho_{a-1}\Psi_{a-1},\ \ \ a \text{ odd}\\
-\ol{\rho}_{a-1}\Psi_{a-1}, \ \ \ \qquad\qquad\qquad\qquad\ \,  a \text{ even}
\end{cases}\\
&-G_{[a,b],z}(y,b) \cdot \begin{cases}
-\rho_b\Psi_{b+1}, \ \  \qquad\qquad\qquad\qquad\qquad b \text{ odd}\\
-z(\frac{1}{\ol{\alpha}_b}-\alpha_b)\Psi_b+z\ol{\rho}_b\Psi_{b+1}, \ \ \qquad \ \  b \text{ even}
\end{cases}
\end{align}

Combining \eqref{eq:Cramer1} with \eqref{eq:Poisson}, we have for odd $y$, for our $\alpha_n$ in \eqref{def:alpha_n} and $z\in \sigma(W_{\lambda_1,\lambda_2,\omega,\theta})$,
\begin{align}\label{eq:Poisson_1}
|\Psi_y|\lesssim_{\lambda_1,\lambda_2} 
&(\prod_{j=a}^{y-1} |\rho_j|) \frac{|P_{[y+1,b],z}|}{|P_{[a,b],z}|}\max(|\Psi_{a-1}|, |\Psi_a|)\\
&+(\prod_{j=y}^{b-1} |\rho_j|) \frac{|P_{[a,y-1],z}|}{|P_{[a,b],z}|}\max(|\Psi_b|, |\Psi_{b+1}|) \notag
\end{align}
One can also prove that \eqref{eq:Poisson_1} also holds for even $y$.

\subsection{Boundary Conditions}
\label{ss: Bdry}
For $\beta,\gamma\in \C$, define
\begin{align}\label{def:modify}
  \tilde{\alpha}_n=
  \begin{cases} \alpha_n,\quad &n\neq a-1,b\\
    \beta,\quad &n = a-1\\
    \gamma,\quad &n = b\\
  \end{cases}.
\end{align}
Let $\tilde{W}$ be the CMV matrix with Verblunsky coefficient $\tilde{a}_n$.
Let $W_{[a,b]}^{\beta,\gamma}=\chi_{[a,b]}\tilde{W}\chi_{[a,b]}$, and define $\mathcal{L}^{\beta,\gamma}_{[a,b]}$ and $\mathcal{M}^{\beta,\gamma}_{[a,b]}$ similarly.
Let
\begin{align}\label{eq:Wcdot}
W_{[a,b]}^{\beta,\cdot}:=W_{[a,b]}^{\beta,\alpha_b}, \text{ and } W_{[a,b]}^{\cdot,\gamma}:=W_{[a,b]}^{\alpha_{a-1},\gamma},
\end{align}
be the operator with only one-sided boundary condition. Define \eqref{eq:Wcdot} similarly for $\mathcal{L}$ and $\mathcal{M}$.
It is easy to check that $W_{[a,b]}^{\beta,\gamma}=\mathcal{L}_{[a,b]}^{\beta,\gamma}\mathcal{M}_{[a,b]}^{\beta,\gamma}$.

Let 
\begin{align*}
\mathcal{A}^{\beta,\gamma}_{[a,b],z}:=
(z (\mathcal{L}^{\beta,\gamma}_{[a,b]})^{-1}-\mathcal{M}^{\beta,\gamma}_{[a,b]}).
\end{align*} 
It is easy to check that it is a tri-diagonal matrix. Let
\begin{align*}
G^{\beta,\gamma}_{[a,b]}:=(\mathcal{A}^{\beta,\gamma}_{[a,b]})^{-1}
\end{align*}
be the Green's function.


Let $P_{[a,b]}^{\beta,\gamma}(\theta):= \det(z-W^{\beta,\gamma}_{[a,b]})$.
The following was proved in \cite{SOPUC}
\begin{align}\label{eq:Sze1}
S_{b,z}\cdots S_{a,z}=\frac{1}{2 \prod_{j=a}^b |\rho_j|} \left(\begin{matrix}
P_{[a,b],z}^{-1,\cdot}+P_{[a,b],z}^{1,\cdot}\ \ &P_{[a,b],z}^{-1,\cdot}-P_{[a,b],z}^{1,\cdot}\\
(P_{[a,b],z}^{-1,\cdot}-P_{[a,b],z}^{1,\cdot})^* &(P_{[a,b],z}^{-1,\cdot}+P_{[a,b],z}^{1,\cdot})^*
\end{matrix}\right).
\end{align}
An alternate expression was given in \cite{Wang}:
\begin{align}\label{eq:Sze2}
S_{b,z}\cdots S_{a,z}=\frac{1}{\prod_{j=a}^b |\rho_j|}
\left(\begin{matrix}
zP_{[a+1,b]} &\frac{zP_{[a+1,b],z}-P_{[a,b],z}}{\alpha_{a-1}}\\
z(\frac{zP_{[a+1,b],z}-P_{[a,b],z}}{\alpha_{a-1}})^*& (P_{[a+1,b]})^*
\end{matrix}\right).
\end{align}


\section{Equivalence between the Gesztesy--Zinchenko cocycle and the transfer matrix cocycle}\label{sec:equiv}
Since $v=\mathcal{L}^{-1}u$ as in \eqref{def:uv}, we have
\begin{align}\label{eq:tran_1}
\begin{cases}
v_{2n}=\alpha_{2n}u_{2n}+\rho_{2n}u_{2n+1}\\
v_{2n+1}=\ol{\rho}_{2n} u_{2n}-\ol{\alpha}_{2n}u_{2n+1}
\end{cases}
\end{align}
Hence by \eqref{eq:tran_1} and \eqref{eq:GZ_transfer},
\begin{align*}
\left(\begin{matrix}
1 &0\\
-\ol{\alpha}_{2n} &\ol{\rho}_{2n}
\end{matrix}\right)
\left(\begin{matrix}
u_{2n+1}\\
u_{2n}
\end{matrix}\right)=
\left(\begin{matrix}
u_{2n+1}\\
v_{2n+1}
\end{matrix}\right)=&M_{2n,z}M_{2n-1,z}
\left(\begin{matrix}
u_{2n-1}\\
v_{2n-1}
\end{matrix}\right)\\
=&M_{2n,z}M_{2n-1,z}
\left(\begin{matrix}
1 &0\\
-\ol{\alpha}_{2n-2} &\ol{\rho}_{2n-2}
\end{matrix}\right)\left(\begin{matrix}
u_{2n-1}\\
u_{2n-2}
\end{matrix}\right).
\end{align*}
Hence we have
\begin{align}\label{eq:A=MM}
A_{n,z}=B_n^{-1} M_{2n,z} M_{2n-1,z} B_{n-1},
\end{align}
where
\begin{align*}
B_n:=\left(\begin{matrix}
1 &0\\
-\ol{\alpha}_{2n} &\ol{\rho}_{2n}
\end{matrix}\right).
\end{align*}

Comparing \eqref{eq:MM=SS} with \eqref{eq:A=MM}, we see that the three cocycles $(\omega, \tilde{S})$, $(\omega, \tilde{M})$ and $(\omega, A)$ are conjugate to each other, in particular, their Lyapunov exponents are identical:
\begin{align}\label{eq:LA=LS=LM}
L(\omega, \tilde{S}(\cdot, z))\equiv L(\omega, \tilde{M}(\cdot, z))\equiv L(\omega, A(\cdot, z)).
\end{align}

\section{The key lemma}\label{sec:key}
For our $\alpha_n$ in \eqref{def:alpha_n}, we have that
\begin{lemma}\label{lem:poly1}
$P_{[1,2n],z}(\theta)$ is a polynomial of $\cos(2\pi\theta)$ and $\sin(2\pi\theta)$ of degree at most $n$.
\end{lemma}
\begin{proof}

by \eqref{eq:SS} that,
\begin{align*}
\tilde{S}(\theta,z)=\frac{1}{z}S_{1,z}S_{0,z}
=&\frac{1}{|\rho_1\rho_0|}\left(\begin{matrix}
\lambda_1'\lambda_2\sin(2\pi\theta)+z &-\lambda_1'-\lambda_2\sin(2\pi\theta) z^{-1}\\
-\lambda_1'-\lambda_2\sin(2\pi\theta) z &\lambda_1'\lambda_2\sin(2\pi\theta)+z^{-1}
\end{matrix}\right)\\
=:&\frac{1}{|\rho_1\rho_0|}D(\theta,z).
\end{align*}
Combining this with \eqref{eq:Sze2}, we have that $P_{[1,2n],z}$ is the upper left corner of 
\begin{align}\label{eq:P_upper_left}
z^{n-1} S_{2n,z} \prod_{j=n-1}^{0} D(\theta+j\omega),
\end{align}
hence it is a polynomial in $\cos(2\pi\theta)$ and $\sin(2\pi\theta)$ of degree at most $n$, where we used that $S_{2n,z}$ is constant in $\theta$ and each entry of $D(\theta+j\omega)$ has degree at most $1$.
\end{proof}

\begin{lemma}\label{lem:poly2}
$P_{[1,2n],z}(\theta+\frac{1}{4}-\frac{n-1}{2}\omega)$ is an even function in $\theta$.
\end{lemma}
\begin{proof}
We have
\begin{align*}
P_{[1,2n],z}(\theta+\frac{1}{4}-\frac{n-1}{2}\omega)=\det(\mathcal{M}|_{[1,2n]}(\theta))\det(\tilde{\mathcal{A}}_{[1,2n],z}(\theta+\frac{1}{4}-\frac{n-1}{2}\omega)).
\end{align*}
Note that $\det(\mathcal{M}|_{[1,2n]}(\theta))\equiv 1$, hence it suffices to consider
\begin{align*}
\det(\tilde{\mathcal{A}}_{[1,2n],z}(\theta+\frac{1}{4}-\frac{n-1}{2}\omega))=:\det(R(\theta)).
\end{align*}
For $k\in \frac{1}{2}\Z$, denoting $\sin(2\pi(\theta+k\omega))=:s_k(\theta)$ and $\cos(2\pi(\theta+k\omega))=:c_k(\theta)$.
Let 
\begin{align*}
C:=
\left(\begin{matrix}
0\ \ &0\\
-\lambda_1 &0
\end{matrix}\right),
\end{align*}
and
\begin{align*}
B_k(\theta):=
\left(\begin{matrix}
z\lambda_2c_k(\theta)+\lambda_1'  &z(-\lambda_2 s_k(\theta)+i\lambda_2')\\
z(-\lambda_2 s_k(\theta)-i\lambda_2') &-z\lambda_2c_k(\theta)-\lambda_1'
\end{matrix}\right)
\end{align*}
Let
\begin{align*}
L_1:=\left(\begin{matrix}
0  & 1\\
1 &0
\end{matrix}\right).
\end{align*}
We note that 
\begin{align}\label{eq:L1BL1}
L_1B_k(\theta)L_1=&
\left(\begin{matrix}
-z\lambda_2c_k(\theta)-\lambda_1' &z(-\lambda_2 s_k(\theta)-i\lambda_2')\\
z(-\lambda_2 s_k(\theta)+i\lambda_2') &z\lambda_2c_k(\theta)+\lambda_1'
\end{matrix}\right)=-B_{-k}(-\theta).
\end{align}
Let 
\begin{align*}
L_2:=\left(\begin{matrix}
& & & & &I_2\\
& & & &I_2 &\\
& & & & &\\
& &I_2 & & &\\
&I_2 & & & &
\end{matrix}\right).
\end{align*}
Then it is clear that
\begin{align*}
L_2R(\theta)L_2=
\left(
\begin{matrix}
&B_{\frac{n-1}{2}}\ \ &C^T\ \ & & & & &\\
&C &B_{-1+\frac{n-1}{2}} &\ddots & & & &\\
& &\ddots &\ddots & & & &\\
& & & &  & B_{1-\frac{n-1}{2}} &C^T &\\
& & & & &C &B_{-\frac{n-1}{2}}
\end{matrix}
\right)
\end{align*}
Let 
\begin{align*}
L_3:=
\left(\begin{matrix}
&(-1)^1 L_1 & & & &\\
& &(-1)^2 L_1 & & &\\
& & &\ddots & &\\
& & & &(-1)^{n-1} L_1 &\\
& & & & &(-1)^n L_1
\end{matrix}\right)
\end{align*}
We then have
\begin{align*}
&L_3L_2R(\theta)L_2L_3\\=
&\left(
\begin{matrix}
&L_1B_{\frac{n-1}{2}}(\theta)L_1\ \ &-L_1C^TL_1\ \ & & &\\
&-L_1CL_1 &L_1B_{-1+\frac{n-1}{2}}(\theta)L_1 &\ddots  & &\\
& &\ddots &\ddots  & &\\
& &  & &L_1B_{1-\frac{n-1}{2}}(\theta)L_1 &-L_1C^TL_1 \\
& &  & &-L_1CL_1 &L_1B_{-\frac{n-1}{2}}(\theta)L_1
\end{matrix}
\right)
\end{align*}
By \eqref{eq:L1BL1} and $L_1CL_1=C^T$, we have
\begin{align*}
L_3L_2R(\theta)L_2L_3=-R(-\theta),
\end{align*}
which implies $\det(R(\theta))$ is an even function in $\theta$. 
\end{proof}
Combining Lemmas \ref{lem:poly1} and \ref{lem:poly2}, we arrive at the key lemma.
\begin{lemma}\label{lem:main}
$P_{[1,2n],z}(\theta)$ is a polynomial of $\sin(2\pi(\theta+\frac{n-1}{2}\omega))$ of degree at most $n$.
\end{lemma}

Next, we obtain the average lower bound of $P_{[1,2n],z}(\theta)$.
\begin{lemma}\label{lem:ave_low}
We have for any $\varepsilon>0$, for $n$ large enough, the following holds
\begin{align*}
\frac{1}{2n}\int \ln|P_{[1,2n],z}(\theta)|\, d\theta\geq \frac{1}{2}\ln(\lambda_2(1+\lambda_1'))-\varepsilon.
\end{align*}
\end{lemma}
\begin{proof}
This proof uses Herman's trick \cite{Herman}.
Let $U:=\mathrm{diag}(1,0)$.
Let 
\begin{align*}
\tilde{D}(w):=w^{-1}
\left(\begin{matrix}
\frac{\lambda_1'\lambda_2}{2i}(w^2-1)+wz &-\lambda_1'w-\frac{\lambda_2}{2i}(w^2-1)z^{-1}\\
-\lambda_1'w-\frac{\lambda_2}{2i}(w^2-1)z &\frac{\lambda_1'\lambda_2}{2i}(w^2-1)+wz^{-1}
\end{matrix}\right)=:w^{-1}\tilde{D}_1(w),
\end{align*}
where $\tilde{D}(w)$ is just $D(\theta)$ with $e^{2\pi i\theta}$ replaced with $w$.
By \eqref{eq:P_upper_left}, we have
\begin{align}\label{eq:Herman1}
\frac{1}{2n}\int_{\T} \ln|P_{[1,2n],z}(\theta)|\, d\theta
=&\frac{1}{2n}\int_{\T} \ln \left\|U(S_{2n,z} \prod_{j=n-1}^0 D(\theta+j\omega))U\right\|\, d\theta \notag\\
=&\frac{1}{2n}\int_{\partial \mathbb D} \ln \left\| U(S_{2n,z} \prod_{j=n-1}^0 \tilde{D}_1(e^{2\pi i j\omega} w)) U\right\|, dw \notag\\
\geq &\frac{1}{2n} \ln \left\|U(S_{2n,z} (\tilde{D}_1(0))^n)U\right\|,
\end{align}
where we used the property of sub-harmonic functions.

It is easy to compute
\begin{align}\label{eq:D10}
\tilde{D}_1(0)=\frac{1}{2i}
\left(\begin{matrix}
-\lambda_1'\lambda_2 &\lambda_2z^{-1}\\
\lambda_2 z &-\lambda_1'\lambda_2
\end{matrix}\right)=
\frac{1}{2i} Y^{-1}\mathrm{diag}(\lambda_2(1+\lambda_1'), \lambda_2(-1+\lambda_1'))Y,
\end{align}
where 
\begin{align*}
Y=\left(\begin{matrix}
1 &1\\
z &-z
\end{matrix}\right).
\end{align*}
Plugging \eqref{eq:D10} into \eqref{eq:Herman1}, we have
\begin{align*}
\frac{1}{2n}\int_{\T} \ln |P_{[1,2n],z}(\theta)|\, d\theta\geq \frac{1}{2n}\ln|\frac{(-z^2+\lambda_1'z)a^n-(z+\lambda_1')b^n}{2}|\geq \frac{1}{2}\ln\frac{\lambda_2(1+\lambda_1')}{2}-\varepsilon,
\end{align*}
where $a=\lambda_2(1+\lambda_1')$ and $b=\lambda_2(-1+\lambda_1')$. This proves Lemma \ref{lem:ave_low}.
\end{proof}

\section{Anderson localization: Proof of Theorem \ref{thm:main}}
With our preparations in Section \ref{sec:key}, the proof of Theorem \ref{thm:main} is very similar to that of the almost Mathieu model in \cite{jit99}, or more recent works on the other related models \cite{JKS,jyMaryland,ehmtransition,ew}. We only give an outline of proof below.

By Schnol's theorem \cite{berez,DFLY,HanSchnol}, to prove Anderson localization, it suffices to show generalized eigenfunction $\Psi$, namely 
$|\Psi_y|\leq C|y|$ for some constant $C$, to $W_{\lambda_1,\lambda_2,\omega,\theta}\Psi=z\Psi$ decays exponentially.
Throughout this section, let 
\begin{align*}
\begin{cases}
L_+:=\ln\frac{\lambda_2(1+\lambda_1')}{2}\\
L_-:=\ln\frac{\lambda_1(1+\lambda_2')}{2}\\
L:=L_+-L_-=\ln\frac{\lambda_2(1+\lambda_1')}{\lambda_1(1+\lambda_2')}
\end{cases}
\end{align*}

Without loss of generality, we consider $y>0$ sufficiently large. 
Let $n$ be the smallest positive integer such that $\varepsilon q_n<y<\frac{1}{20}q_{n+1}$, where $p_n/q_n$ is the continued fraction approximant to $\omega$, see \eqref{defpnqn}.
Let $m$ be the largest positive integer such that $y\geq 6q_m$ and let $s$ be the largest integer such that 
\begin{align}\label{eq:s_qm}
\max(\varepsilon q_n, 6sq_m)\leq y<\min(6(s+1)q_m, 6q_{m+1}).
\end{align}

Let $h=2sq_m-2$ and set 
\begin{align*}
\begin{cases}
I_1:=[y-sq_m-[sq_m/4]+1, y-[sq_m/4]]\cap (2\Z+1)\\
I_2:=[-2sq_m+[sq_m/4]+1, -sq_m+[sq_m/4]]\cap (2\Z+1)
\end{cases}
\end{align*}

The following lemma essentially goes back to \cite{jit99}, and is proved in exactly this form in \cite{ew}.
\begin{lemma}\label{lem:uni}
For $\theta$ such that $\tilde{\gamma}(\omega,\theta)=0$, 
we have $\{\theta+j\omega-\frac{1}{4}:\ j\in I_1\cup I_2\}$ is $\varepsilon$-uniform when $y$ is sufficiently large.
\end{lemma}
Note that $I_1\cup I_2$ consist of at least $h+1$ integers.

Combining Lemmas \ref{lem:uni} and \ref{lem:ave_low} with \eqref{eq:Lagrange}, we have the following.
\begin{lemma}\label{lem:I1_I2_large}
There exists $x_1\in I_1\cup I_2$ such that 
\begin{align}\label{eq:P_den_large}
|P_{[x_1,x_1+h-1],z}(\theta)|\geq e^{(\frac{L_+}{2}-2\varepsilon)h}.
\end{align}
\end{lemma}

It follows from a standard argument that Lemma \ref{lem:I1_I2_large} implies the following.
\begin{corollary}\label{cor:I2_large}
There exists $x_1\in I_2$ such that \eqref{eq:P_den_large} holds.
\end{corollary}
Let $x_2:=x_1+h-1$.
Combining Lemma \ref{lem:upperbounds}, Theorem \ref{thm:LE} with \eqref{eq:LA=LS=LM}, \eqref{eq:int_rho}, \eqref{eq:scalar_upper} and \eqref{eq:Sze2}, we have
\begin{align}\label{eq:nu_1}
\prod_{j=y+1}^{x_2}|\rho_j|\cdot |P_{[x_1,y-1],z}(\theta)|
\leq &\prod_{j=y+1}^{x_2} |\rho_j| \prod_{j=x_1-1}^{y-1} |\rho_j|\cdot  \|S_{y-1,z}\cdots S_{x_1-1,z}\|\\
\leq &e^{\frac{h}{2}(L_-+\varepsilon)} e^{\frac{y-x_1}{2}(L+\varepsilon)}
=e^{\frac{y-x_1}{2}(L_++2\varepsilon)} e^{\frac{x_2-y}{2}(L_-+\varepsilon)},
\end{align}
and similarly
\begin{align}\label{eq:nu_2}
\prod_{j=x_1}^{y-1} |\rho_j| \cdot |P_{[y+1,x_2],z}(\theta)|\leq e^{\frac{x_2-y}{2}(L_++2\varepsilon)} e^{\frac{y-x_1}{2}(L_-+\varepsilon)}.
\end{align}

Finally, combining \eqref{eq:P_den_large}, \eqref{eq:nu_1}, \eqref{eq:nu_2} with \eqref{eq:Poisson},
\begin{align}\label{eq:nu_3}
\ \ \ \ |\Psi_y|\leq \max(e^{-(\frac{L}{2}-25\varepsilon)|y-x_1|} \max(|\Psi_{x_1-1}|, |\Psi_{x_1}|), e^{-(\frac{L}{2}-25\varepsilon)|y-x_2|} |\max(|\Psi_{x_2}|,|\Psi_{x_2+1}|)).
\end{align}
Starting with a $y\in (\frac{1}{40}q_n, \frac{1}{40}q_{n+1})$, we iterate \eqref{eq:nu_3} until we reach at $x_1\leq \varepsilon q_n$ or $x_2\geq \frac{1}{20}q_{n+1}$ or the iteration number reaches $C/\varepsilon$, we have by a standard argument that
\begin{align*}
|\Psi_y|\leq e^{-(\frac{L}{2}-30\varepsilon)y}.
\end{align*}
This shows the upper bound in \eqref{eq:decay_rate}.
The lower bound of \eqref{eq:decay_rate} is standard.
This proves Theorem \ref{thm:main}. \qed

\bibliographystyle{amsplain}

\end{document}